\documentclass[12pt,twoside]{amsart}
\usepackage{amsmath}
\usepackage{amsthm}
\usepackage{amsfonts}
\usepackage{amssymb}
\usepackage{latexsym}
\usepackage{mathrsfs}
\usepackage{amsmath}
\usepackage{amsthm}
\usepackage{amsfonts}
\usepackage{amssymb}
\usepackage{latexsym}
\usepackage{geometry}
\usepackage{dsfont}
\usepackage[dvips]{graphicx}
\usepackage{color}
\usepackage[all]{xy}

\date{}
\allowdisplaybreaks[4] \footskip=15pt
\renewcommand{\uppercasenonmath}[1]{}

\usepackage{graphicx,amssymb}
\usepackage[all]{xy}
\usepackage{amsmath}

\allowdisplaybreaks
\usepackage{amsthm}
\usepackage{color}

\theoremstyle{plain}
\newtheorem*{Conflict of interest}{Conflict of interest}
\newtheorem*{Data Availability}{Data Availability}
\newtheorem{theorem}{Theorem}[section]

\newtheorem{lemma}[theorem]{Lemma}
\newtheorem{corollary}[theorem]{Corollary}
\newtheorem{example}[theorem]{Example}
\newtheorem{hypothesis}[theorem]{Hypothesis}
\newtheorem*{open question}{Open Question}
\newtheorem{definition}[theorem]{Definition}
\newtheorem*{question}{Question}
\theoremstyle{definition}

\theoremstyle{remark}

\newcommand{\M}{\mathcal{M}}

\newcommand{\Id}{\mathrm{Id}}

\def\s{\frak s}

\def\Hom{{\rm Hom}}

\def\Ker{{\rm Ker}}

\def\Im{{\rm Im}}

\def\Mod{{\rm Mod}}

\begin{document}
\begin{center}
{\large  \bf The Schr\"{o}der-Bernstein problem for relative injective modules}

\vspace{0.5cm}   Xiaolei Zhang$^{a}$

{\footnotesize
	School of Mathematics and Statistics, Shandong University of Technology,
	Zibo 255049, China\\
	
	E-mail: zxlrghj@163.com\\}

\end{center}

\bigskip
\centerline { \bf  Abstract}
\bigskip
\leftskip10truemm \rightskip10truemm \noindent

Let $(K,\M)$  be a pair satisfying some mild condition, where $K$ is a class of $R$-modules and $\M$ is a class of $R$-homomorphisms. We show that if $f:A\rightarrow B$ and $g:B\rightarrow A$ are $\M$-embeddings and $A,B$ are $K_\M$-injective, then $A$ is isomorphic  to $B$, positively answering an question proposed by Marcos and Jiri \cite{MR24}.	
\vbox to 0.3cm{}\\
{\it Key Words:}  Schr\"{o}der-Bernstein problem, $K_\M$-injective module, isomorphism.\\
{\it 2020 Mathematics Subject Classification:} 16D50.

\leftskip0truemm \rightskip0truemm
\bigskip

\section{Introduction}
Throughout this paper, $R$ always is a fixed ring with unit, and all modules are unitary.

The  Schr\"{o}der-Bernstein theorem, which states that if $A$ and $B$ are two sets and there are two injective maps $f:A\rightarrow B$ and   $g:B\rightarrow A$, then there exists a bijective map between  $A$ and $B,$ is a classical result in basic set theory. This type of problem, called Schr\"{o}der-Bernstein problem, can be extended to various branches of Mathematics. For example, Gowers \cite{G96} showed that there is a negative solution for Banach spaces. In the context of modules,   Bumby \cite{B65} obtained that the Schr\"{o}der-Bernstein problem has a positive solution
for modules which are invariant under endomorphisms of their injective envelope. Asensio {\it et al}.\cite{AKS18} obtained a positive solution for the Schr\"{o}der-Bernstein problem for modules invariant under endomorphisms of their general envelopes under some mild
conditions, including injective envelopes, pure-injective
envelopes, and flat modules  invariant under endomorphisms of cotorsion envelopes.

Recently,  Marcos and  Jiri \cite{MR24} introduced the notion of relative injective modules, which is called $K_\M$-injective modules,  in terms of a given pair $(K,\M)$ where $K$ is a class of $R$-modules and $\M$ is a class of $R$-homomorphisms (see Definition \ref{rinj}). When $K$ satisfying some mild conditions and $\M$ is either the class of embeddings, $RD$-embeddings or pure embeddings (see Hypothsis \ref{hyp} for details), $K_\M$-injective modules will be well-behaved. For example, it has the Baer-like criterion and the superstable property {\it et al}.\cite{MR24}.  A pair $(K,\M)$ is called Noetherian if every direct sum of $K_\M$-injective modules is $K_\M$-injective. The authors \cite{MR24} showed that if $(K,\M)$ satisfies Hypothesis \ref{hyp} and  is Noetherian, then the Schr\"{o}der-Bernstein problem  has a  positive solution for  $K_\M$-injective modules. Subsequently, they proposed the following question:
\begin{question}\cite[Question 5.6]{MR24} Does the previous result still hold even if $(K,\M)$ is not Noetherian?
\end{question}
In this paper, we first study the  existence of  $K_\M$-injective  envelope of $R$-modules in $K$. Utilizing this, we then give a positive answer to this question, that is, the Schr\"{o}der-Bernstein problem always has a  positive solution for all  $K_\M$-injective modules when  $(K,\M)$ satisifes Hypothsis \ref{hyp} (see Theorem \ref{main}).

\section{relative injective envelopes}
Recall from \cite{MR24} that an exact sequence of $R$-modules
$\Psi:0 \rightarrow A \xrightarrow{f} B \xrightarrow{g} C\rightarrow 0$
is said to be pure-exact if every system of linear equations with parameters in $f[A]$, which has a solution
in $B$, has a solution in $f[A].$ In this case we say that $f$ is a pure embedding and $g$ is a pure
epimorphism. We will denote the class of all pure embeddings by $Pure$.  If $f$ is the inclusion, we
say that $A$ is a pure submodule of $B$ and denote it by $A \leq_{Pure} B$.

An exact sequence $\Psi$ of $R$-modules is said to be $RD$-exact if $f[A]\cap rB = rf[A]$ for every $r\in R.$ In that
case we say that $f$ is a $RD$-embedding and $g$ is a $RD$-epimorphism. We will denote the class of
$RD$-embeddings by $RD.$ If $f$ is the inclusion, we say that $A$ is an $RD$-submodule of $B$ and denote
it by $A \leq_{RD} B$.

We denote the class of all embeddings by
$Emb$. If $A$ is a submodule of $B$, we denote
it by $A \leq_{Emb} B$. Note that we always have $Pure\subseteq RD\subseteq Emb.$ But the converse is not  true in general.

 Recently,  Marcos and  Jiri \cite{MR24} introduced the notion of relative injective modules in terms of a given pair $(K,\M)$, where $K$ is a class of $R$-modules and $\M$ is a class of $R$-homomorphisms of $R$-modules in $K$.
\begin{definition}\label{rinj}\cite[Definition 4.1]{MR24}
	Let $(K,\M)$ be a pair such that $K$ is a class of $R$-modules and $\M$ is a class of $R$-homomorphisms of $R$-modules in $K$. An $R$-module $E$ in $K$ is called $K_\M$-injective,  if for every $R$-homomorphism $f:A\rightarrow B$ in $\M$ and every  $R$-homomorphism $g:A\rightarrow E$ with $A\in K$ there is an $R$-homomorphism $h:B\rightarrow E$ such that the following diagram is commutative:
	$$\xymatrix@R=40pt@C=50pt{
A\ar[r]^{f}\ar[d]_g	&B\ar@{.>}[ld]^{h}\\
	E&}$$		
\end{definition}

Certainly, suppose $f:K\rightarrow M\in \M$, $M\in K$ and  $K$ is $K_\M$-injective. Then $K$ is a direct summand of $M$. 

\begin{example}\label{examp} \cite[Example 4.3]{MR24} We denote by $_R\Mod$ the class of all $R$-modules, $_RFlat$ the class of all flat $R$-modules,  $_RAbsP$ the class of all absolutely pure $R$-modules, and $\s\mbox{-}Tor$ the class of $\s$-torsion module (see \cite{M20}). Then
\begin{enumerate}
\item If $(K,\M)=(_R\Mod,Emb)$, then the  $K_\M$-injective modules are exactly injective modules;	
\item If $(K,\M)=(_R\Mod,RD)$, then the  $K_\M$-injective modules are exactly  $RD$-injective modules;	
\item If $(K,\M)=(_R\Mod,Pure)$, then the  $K_\M$-injective modules are exactly  Pure-injective modules;
\item If $(K,\M)=(_RFlat,Pure)$, then the  $K_\M$-injective modules are exactly  flat cotorsion modules;
\item If $(K,\M)=(\s\mbox{-}Tor,Pure)$, then the  $K_\M$-injective modules are exactly $K^{\s\mbox{-}Tor}$-pure injective modules (see \cite{MR23});
\item If $(K,\M)=(_RAbsP,Emb)$, then the  $K_\M$-injective modules are exactly injective modules.	
\end{enumerate}	
\end{example}

For further study of $K_\M$-injective modules, the authors in \cite{MR24} proposed the following Hypothesis on $K$ and $M$.

\begin{hypothesis}\label{hyp} Let $(K,\M)$ be a pair such that:
	\begin{enumerate}
		\item $K$ is a class of $R$-modules.
		\item $\M$ is either the class of embeddings, $RD$-embeddings or pure embeddings, and 
		\item $K$ is closed under 
		\begin{enumerate}
			\item direct sums,
			\item $\M$-submodules, i.e., if $A\in K$ and $B\leq_\M A$, then $B\in K$, and 
			\item $\M$-epimorphisms, it.e, if $f:A\rightarrow B$ is an $\M$-epimorphism and $A\in K$, then $B\in K$.
		\end{enumerate}
	\end{enumerate}
\end{hypothesis}
Note that all cases in Example \ref{examp} satisfy  Hypothesis \ref{hyp}. 

\begin{lemma}\label{compm} Assume $(K,\M)$ satisfies Hypothesis \ref{hyp}. The the following statements hold:
	\begin{enumerate}
		\item $K$ is closed under direct union of ascending chains.
		\item  Let $g:M\rightarrow N$ and $f:N\rightarrow L$ be $R$-homomorphisms. If the composition $f\circ g\in \M$, then so is $g$.
		\item Let $\{f_\alpha:M\rightarrow N_\alpha\}$ be a family of ascending chains of  $R$-homomorphisms in $\M$. 
		Then $\bigcup f_\alpha: M\rightarrow \bigcup N_\alpha$ is also in $\M$.
		\item Let $g:M\rightarrow N$ and $f:N\rightarrow L$ are embeddings. If the composition $f\circ g$ and the quotient $\overline{f}:N/M\rightarrow L/M$ are in $\M$, then so is $f$. 
			\item Let $\{g_\alpha:M_\alpha\rightarrow N\}$ be a family of ascending chains of  $R$-homomorphisms in $\M$. 
		Then $\bigcup g_\alpha: \bigcup M_\alpha\rightarrow  N	$ is also in $\M$.
	\end{enumerate}
\end{lemma}
\begin{proof}
(1): It follows by any direct union of ascending chains of $R$-modules is a pure quotient (thus an $\M$-quotient) of their direct sums.

(2), (3), (4) and (5): These are well-known for  $\M$ is either the class of embeddings, $RD$-embeddings or pure embeddings.	
\end{proof}

\begin{definition} Let $\sigma:L\rightarrow M\in\M$. Then  $\sigma$ is said to be \emph{$\M$-essential} provided that every  $R$-homomorphism $\tau:M\rightarrow N\in \M$, whenever $\tau\sigma\in\M$.
\end{definition}

\begin{lemma}\label{ess}
Assume $(K,\M)$ satisfies Hypothesis \ref{hyp}.	An $\M$-embedding $\sigma:L\hookrightarrow M$ is  $\M$-essential if and only if every  $\M$-epimorphism $\tau:M\rightarrow N$ is an isomorphism, whenever $\tau\sigma$ is an $\M$-embedding.
\end{lemma}
\begin{proof}
	Let $\tau:M\rightarrow K$ be an $R$-homomorphism such that $L\stackrel{ \sigma}{\hookrightarrow}  M\xrightarrow{\tau} K$ is an  $\M$-monomorphism, i.e., the natural composition $$\delta:L\stackrel{ \sigma}{\hookrightarrow}  M\stackrel{ \pi_\tau}{\twoheadrightarrow} \Im(\tau)\stackrel{ i_\tau}{\hookrightarrow} K$$ is in $\M$,
	It follows by Lemma \ref{compm}(2) that  $L\stackrel{ \sigma}{\hookrightarrow}  M\stackrel{ \pi_\tau}{\twoheadrightarrow} \Im(\tau)$ is also in $\M$.
	By assumption, $\pi_\tau$ is an isomorphism. Consequently,
	$\tau=i_\tau\pi_\tau$ is in $\M$. Hence
	$\sigma$ is  $\M$-essential.	
\end{proof}

Let $M$ be an $R$-module and $\mathscr{C}$ be a  class  of $R$-modules. Recall from \cite{EJ11} that an $R$-homomorphism $f:M\rightarrow C$ with  $C\in \mathscr{C}$ is said to be a  \emph{$\mathscr{C}$-preenvelope}  provided that  for any $C'\in \mathscr{C}$, the natural homomorphism  $\Hom_{R}(C,C')\rightarrow \Hom_{R}(M,C')$ is an epimorphism.
If, moreover, $f$ is left minimal, that is, every endomorphism $h$ such that $f=hf$  is an automorphism,  then $f$ is said to be a  \emph{$\mathscr{C}$-envelope}.

\begin{theorem}\label{equal}
Assume $(K,\M)$ satisfies Hypothesis \ref{hyp}. Let $M\in K$ be an $R$-module which has a $K_\M$-injective preenvelope in $\M$. Suppose $\sigma:M\rightarrow E$ is an $R$-homomorphism with $E$ $K_\M$-injective. Then $\sigma$ is a  $K_\M$-injective envelope if and only if it is an $\M$-essential $\M$-monomorphism.
\end{theorem}
\begin{proof} Suppose $\sigma:M\rightarrow E$ is a  $K_\M$-injective envelope of $M\in K$. Then $\sigma$ is in $\M$.
	We claim that $\sigma$ is $\M$-essential. Indeed, let $\tau:E\rightarrow N$ be an $\M$-epimorphism such that $\tau\sigma\in \M$. Then there is an $R$-homomorphism $\lambda: N\rightarrow E$ as $E$ is $K_\M$-injective.
	
	$$\xymatrix@R=20pt@C=50pt{
		&E\ar@{->>}[d]^{\tau}\\
		M\ar[r]^{\tau\sigma}\ar[ru]^{\sigma}\ar[rd]_{\sigma} &N\ar[d]^{\lambda} \\
		& E\\}$$
	Since $\sigma$ is left minimal, $\lambda\tau$ is an isomorphism. So $\tau$ is in $\M$ by Lemma \ref{compm}(2), and hence $\tau$ is  an isomorphism. Consequently,  $\sigma$ is $\M$-essential by Lemma \ref{ess}.
	
	On the other hand, suppose $\sigma:M\hookrightarrow E\in\M$ is  $\M$-essential  with $E$ an  $K_\M$-injective module. Let  $\varepsilon:M\rightarrow M^{'}\in\M$ be a $K_\M$-injective preenvelope. Then there is an $R$-homomorphism $h:E\rightarrow E'$ such that the following diagram is commutative:$$\xymatrix@R=20pt@C=50pt{M\ar@{^{(}->}[r]^{\sigma}\ar[rd]_{\varepsilon} &E\ar@{.>}[d]^{h} \\
		& E'\\}$$
It is easy to check that $\sigma$ is also a $K_\M$-injective preenvelope.
Next, we will show that $\sigma$ is left minimal. Indeed, let $\tau:E\rightarrow E$ be an  $R$-homomorphism such that $\tau\sigma=\sigma$. Since $\sigma$ is $\M$-essential, $\tau$ is in $\M$. Consider the $\M$-exact sequence $0\rightarrow E\xrightarrow{\tau} E\rightarrow C\rightarrow0$. Since $E$ is $K_\M$-injective, there is an $R$-homomorphism $\tau':E\rightarrow E$ such that $\tau'\tau=\Id_E$. So $E=\Im(\tau)\oplus\Ker(\tau')$. Note that $\Im(\sigma)\subseteq\Im(\tau)$ and we have the following commutative diagram:
	$$\xymatrix@R=20pt@C=50pt{&\Im(\tau)\oplus\Ker(\tau')\ar@{.>}[dd]^{\delta:=\begin{pmatrix}
				1 & 0\\
				0 & 0\\
		\end{pmatrix}}\\	
		M\ar@{^{(}->}[ru]^{\sigma}\ar@{_{(}->}[rd]_{\sigma} &
		\\
		&\Im(\tau)\oplus\Ker(\tau')\\}$$
	Since $\sigma$ is $\M$-essential, $\delta$ is in $\M$ and so $\Ker(\tau')=0$. Hence $E=\Im(\tau)$ which implies that $\tau$ is also an epimorphism, and thus is an isomorphism.  Therefore $\tau$ is left minimal.   Consequently, $\sigma$ is a  $K_\M$-injective envelope of $M$.	
\end{proof}

\begin{theorem}\label{kminjenvel} Assume $(K,\M)$ satisfies Hypothesis \ref{hyp}. Let $M\in K$ be an $R$-module which has a $K_\M$-injective preenvelope in $\M$. Then  $M$ has a $K_\M$-injective envelope.
\end{theorem}
\begin{proof}
 Let $M$ be an $R$-module in $K$. Let  $\varepsilon:M\rightarrow M^{'}\in \M$ be a $K_\M$-injective preenvelope. Then we can see $M$ as an $R$-submodule of $ M^{'}$ by $\varepsilon$. Set $$\Lambda=\{M\stackrel{ \sigma}{\hookrightarrow} N\mid N\leq M^{'},\ N\in K, \sigma\  \mbox{is an}\  \M\mbox{-essential}\ \M\mbox{-embedding}\}.$$	
Since $\Id_M\in \Lambda$, 	$\Lambda$ is not empty. We order $\Lambda$ as follows: let $\sigma_i:M\rightarrow N_i$ and  $\sigma_j:M\rightarrow N_j$ in $\Lambda$. Define $\sigma_i\leq \sigma_j$ if and only if there is an $\M$-embedding $\delta_{ij}$ such that the following diagram is  commutative:	$$\xymatrix@R=20pt@C=50pt{M\ar@{^{(}->}[r]^{\sigma_i}\ar@{_{(}->}[rd]_{\sigma_j} &N_i\ar@{^{(}->}[d]^{\delta_{ij}} \\
	& N_j\\}$$

\textbf{Claim 1: $\Lambda$ has a maximal element} $\boldsymbol{\sigma:M\hookrightarrow N}$\textbf{.} Indeed, let $\Gamma=\{\sigma_i\mid i\in A\}$ be a total subset of 	$\Lambda$. Then $\{N_i,\delta_{ij}\mid i\leq j\in A\}$ is a directed system of $R$-modules. Set $N=\bigcup\limits_{i\in A}N_i$ and $\sigma=\bigcup\limits_{ A}\sigma_i:M\rightarrow N$. 
It follows by Lemma \ref{compm}(1,2) that $N\in K$ and $\sigma$ is also an $\M$-embedding. We will show that $\sigma$ is $\M$-essential. Let $\tau:N\rightarrow L$ be an $R$-homomorphism such that  composition map $\phi:M\stackrel{ \sigma}{\hookrightarrow} N\xrightarrow{\tau} L$ is in $\M$.  For each $i,j\in A$, consider the following commutative diagram:

$$\xymatrix@R=30pt@C=60pt{&N_i\ar@{_{(}->}[d]_{\delta_{ij}}\ar[rdd]^{\tau\delta_i}\ar@{^{(}->}@/^1pc/[dd]^{\delta_{i}}&\\ &N_j\ar@{_{(}->}[d]_{\delta_j}\ar[rd]^{\tau\delta_j}&\\
	M\ar[ruu]^{\sigma_i}\ar[ru]^{\sigma_j}\ar[r]_{\sigma}&N\ar[r]_{\tau}	& L\\}$$
Since each $\sigma_{i}$ is $\M$-essential and $\tau\delta_{i}\sigma_{i}=\tau\sigma$ is in $\M$, we have $\tau\delta_{i}$ is also in $\M$.  Hence $\tau=\bigcup\limits_{A} \tau\delta_{i}$ is also in $\M$  by Lemma \ref{compm}(2) again. Consequently,  $\sigma$ is $\M$-essential. It follows that  $\sigma$ is an upper bound of $\Gamma$ in $\Lambda$.
It follows by Zorn Lemma that $\Lambda$ has a maximal element. We also it denoted by $\sigma:M\hookrightarrow N.$

\textbf{Claim 2: $\boldsymbol{N}$ is a $\boldsymbol{K_\M}$-injective module.} First, we will show $N$ has no nontrivial $\M$-essential $\M$-embedding. Indeed, let $\mu:N\hookrightarrow X$ be an $\M$-essential $\M$-embedding. Since $ M^{'}$ is $K_\M$-injective, there is an $R$-homomorphism $\xi:X\rightarrow  M^{'}$ such that the following diagram is commutative:
$$\xymatrix@R=30pt@C=60pt{N\ar@{^{(}->}[r]^{\mu}\ar@{_{(}->}[rd]_{i} &X\ar@{.>}[d]^{\xi} \\
	& M^{'}\\}$$
Since $\mu$ is  $\M$-essential  and $i=\xi\mu$ is in $\M$, we have $\xi$ is in $\M$. Let $\beta:M\stackrel{ \sigma}{\hookrightarrow}N\stackrel{ \mu}{\hookrightarrow}X\xrightarrow[\cong]{\pi_\xi}\xi(X)$ be the composition map. Then $\beta\in\Gamma$ and $\beta\geq \sigma$. Since $\sigma$ is maximal in $\Gamma$, we have $N=X$. Consequently, $N$ has no nontrivial $\M$-essential $\M$-embedding. Now, let $\Psi$ the set of $\M$-submodules $H\subseteq  M^{'}$ with $H\cap N=0$. Since $0\in \Psi$, $\Psi$ is nonempty.  Order $\Psi$ as follows $H_i\leq H_j$ if and only if $H_i$ is an $\M$-submodule of  $H_j$. Next, we will show $\Psi$ has a maximal element. Let $\Phi=\{H_i\mid i\in B\}$ be a total ordered subset of $\Psi$. Consider the following commutative diagram:
$$\xymatrix@R=30pt@C=100pt{H_i\ar@{^{(}->}[d]^{\delta_{ij}}\ar[rdd]^{\alpha_i}\ar@{_{(}->}@/_1pc/[dd]_{\delta_{i}}&\\ H_j\ar@{^{(}->}[d]^{\delta_j}\ar[rd]^{\alpha_j}&\\
	\bigcup\limits_{ B}H_i\ar[r]^{\alpha}	&  M^{'}\\}$$
Since each $\alpha_i$ is in $\M$, $\alpha=\bigcup\limits_{ B}\alpha_i$ is in $\M$ by  Lemma \ref{compm}(5). It is easy to verify $ H'\cap N=0$. Hence,  $ H'\in\Psi$ is an upper bound of $\Phi$. It follows by Zorn Lemma that  $\Psi$ has a maximal element, which is also denoted by $H'$. Then $H'$ is an $\M$-submodule of $ M^{'}$ with $H'\cap N=0$.  So the canonical map $\pi: M^{'}\twoheadrightarrow  M^{'}/H'$ induces  in $\M$
$\pi|_N:N\hookrightarrow M^{'}/H'$. Then, we will show $\pi|_N$ is $\M$-essential.
Let $\nu: M^{'}/H'\twoheadrightarrow  M^{'}/H$ be a  natural $\M$-epimorphism such that $\nu\pi_N$ is in $\M$, that is, the following diagram is commutative:
$$\xymatrix@R=30pt@C=60pt{N\ar@{^{(}->}[r]^{\pi|_N}\ar@{_{(}->}[rd]_{\nu\pi|_N} & M^{'}/H'\ar@{->>}[d]^{\nu} \\
	& M^{'}/H\\}$$
Then $H\cap N=0,$ and $H$, which contains $H'$, is an $\M$-submodule of $M^{'}$ by  Lemma \ref{compm}(4). By the maximality of $H'$, we have $H=H'$, and hence $\nu$ is an identity.
Consequently, $\pi|_N$ is $\M$-essential by Lemma \ref{ess}. Finally, we will show $N$ is $K_\M$-injective.  Since $N$  has no nontrivial $\M$-essential $\M$-monomorphism as above,
we have $\pi|_N$ is actually an isomorphism. Thus the exact sequence
$$0\rightarrow H'
\rightarrow
M^{'}
\xrightarrow{(\pi|_N)^{-1}\pi} N\rightarrow 0$$
splits. It follows that $N$ is a direct summand of $ M^{'}$. Consequently, $N$ is a $K_\M$-injective module.

In conclusion,  it follows by Theorem \ref{equal} that $\sigma:M\rightarrow N$ is a $K_\M$-injective envelope.		
\end{proof}

\section{Schr\"{o}der-Bernstein properties}

It is well-known that a ring is Noetherian if and only if every direct sum of injective modules is injective. 

\begin{definition}\cite[Definition 5.2]{MR24} Assume $(K,\M)$ is a pair such that $K$ is a class of $R$-modules and $\M$ is a class of $R$-homomorphisms. Then $(K,\M)$ is called Noetherian provided that every direct sum of $K_\M$-injective modules is $K_\M$-injective.	
\end{definition}	

The well-known Schr\"{o}der-Bernstein theorem for injective modules states that if $f:A\rightarrow B$ and $g:B\rightarrow A$ are embeddings and $A,B$ are injective, then $A$ is isomorphic  to $B$ (see \cite{AKS18}). The authors in \cite{MR24} obtained the Schr\"{o}der-Bernstein theorem for relative injective modules under the Noetherian condition and the Hypothesis \ref{hyp}.

\begin{lemma} \cite[Lemma 5.5]{MR24} Assume $(K,\M)$ satisfies Hypothesis \ref{hyp} and $(K,\M)$ is Noetherian.   If $f:A\rightarrow B$ and $g:B\rightarrow A$ are $\M$-embeddings and $A,B$ are $K_\M$-injective, then $A$ is isomorphic  to $B$.	
\end{lemma}

Subsequently, they proposed the following question:

\begin{question}\cite[Question 5.6]{MR24} Does the previous result still hold even if $(K,\M)$ is not Noetherian?
\end{question}

The main result of this paper is to give a positive answer to this question.

\begin{theorem}\label{main} Assume $(K,\M)$ satisfies Hypothesis \ref{hyp}.   If $f:A\rightarrow B$ and $g:B\rightarrow A$ are $\M$-embeddings and $A,B$ are $K_\M$-injective, then $A$ is isomorphic  to $B$.	
\end{theorem}
\begin{proof} Let $A$ and $B$ be $K_\M$-injective modules, and $f:A\rightarrow B$, $g:B\rightarrow A$ be $\M$-embeddings. Then  $A, B\in K.$  Note that we have $A=C_0\oplus B_1$ and $B_1=D_0\oplus A_1$, where $B_1=g(B)\cong B$ and $A_1=gf(A)\cong A$. From the isomorphism $=gf: A\rightarrow A_1$, we have $A_1=C_1\oplus B_2$ with $C_1=gf(C_0)\cong C_0$ and $B_2=gf(B_1)\cong B_1\cong B_0$. Similarly, $B_2=D_1\oplus A_2$, with $A_2=gf(A_1)\cong A_1$ and $D_1=gf(D_0)\cong D_0$. Inductively, we can construct sequences $\{A_m\}_{m=1}^\infty$, $\{B_m\}_{m=1}^\infty$, $\{C_m\}_{m=0}^\infty$, $\{D_m\}_{m=0}^\infty$ of $R$-modules such that $$A=(\bigoplus\limits_{m=0}^n C_m)\oplus(\bigoplus\limits_{m=0}^n D_m)\oplus A_{n+1}$$ 
for each $n\geq 0$, and $A\cong A_n, B\cong B_n, C_0\cong C_n, D_0\cong D_n$ for each $n\geq 1,$	 and 
$$B_1=(\bigoplus\limits_{m=1}^n C_m)\oplus(\bigoplus\limits_{m=0}^{n-1} D_m)\oplus B_{n+1}$$ 
for each $n\geq 1.$	
	
Write $$E_n=(\bigoplus\limits_{m=1}^n C_m)\oplus(\bigoplus\limits_{m=0}^{n-1} D_m)$$ 	for each $n\geq 1.$	Then $\{E_n\}_{n=1}^\infty$ is an increasing sequence direct summands, and thus $\M$-submodules, of both $A$ and $B_1$. Set $$E=\bigcup\limits_{n=1}^\infty E_n.$$ Then $E$ is a pure submodule, and thus an $\M$-submodule, of both $A$ and $B_1$. So $E$ is in $K$. Since $A$  is $K_\M$-injective, so $A$ is a $K_\M$-injective preenvelope of $E$. It follows by Theorem \ref{kminjenvel} there exist a $K_\M$-injective envelope, denoted by $K(E)$, of $E\in K$. Consequently,  there exists $X\leq A$ and $Y\leq B_1$ such that $A=K(E)\oplus X$ and $B_1=K(E)\oplus Y$.

 Note that  $$E\cong (\bigoplus\limits_{m=1}^\infty C_m)\oplus(\bigoplus\limits_{m=0}^\infty D_m)\cong  (\bigoplus\limits_{m=0}^\infty C_m)\oplus(\bigoplus\limits_{m=0}^\infty D_m)\cong C_0\oplus E,$$ since $C_0\cong C_m$ for each $m\geq 1.$	
It follows by \cite[Corollary 6.4.4]{EJ11} that $K(E)\cong K(C_0\oplus E)\cong C_0\oplus K(E)$, where $K(C_0\oplus E)$ is the $K_\M$-injective envelope of $C_0\oplus E$. As a consequence, we have 
$$A=C_0\oplus B_1=C_0\oplus (K(E)\oplus Y)\cong (C_0\oplus K(E))\oplus Y\cong K_p(E)\oplus Y\cong B_1\cong B,$$ 
which states that $A$ is isomorphic to $B$.	
\end{proof}

The next two Corollaries are direct applications of Theorem \ref{main} with Example \ref{examp}.
\begin{corollary}
	If $f:A\rightarrow B$ and $g:B\rightarrow A$ are pure-embeddings $($resp., $RD$-embeddings, embeddings$)$ and $A,B$ are pure-injective $($resp., $RD$-injective, injective$),$ then $A$ is isomorphic  to $B$.	
\end{corollary}

\begin{corollary}
	If $f:A\rightarrow B$ and $g:B\rightarrow A$ are pure-embeddings and $A,B$ are flat cotorsion modules $($or $K^{\s\mbox{-}Tor}$-pure injective modules$),$ then $A$ is isomorphic  to $B$.	
\end{corollary}

\begin{Conflict of interest}
The author declares that he has no conflict of interest.
\end{Conflict of interest}

\begin{Data Availability}
My manuscript has no associated data.
\end{Data Availability}

\end{document}